\documentclass[12pt,a4]{amsart}
\usepackage[a4paper, left=28mm, right=28mm, top=28mm, bottom=34mm]{geometry}
\usepackage[all]{xy}
\usepackage{amsmath}
\usepackage{amssymb}
\usepackage{amsthm}
\usepackage{diagbox}
\usepackage{mathrsfs}
\usepackage{mathtools}
\usepackage{url}
\usepackage{listings}
\theoremstyle{definition}
\newtheorem{thm}{Theorem}[section]
\newtheorem{lem}[thm]{Lemma}

\newtheorem{prop}[thm]{Proposition}

\theoremstyle{definition}

\newtheorem{rem}[thm]{Remark}

\numberwithin{equation}{section}

\def\Q{{\mathbb Q}}
\def\R{{\mathbb R}}
\def\Z{{\mathbb Z}}
\def\C{{\mathbb C}}

\def\PP{{\mathbb P}}

\def\Br{\mathop{\mathrm{Br}}\nolimits}

\def\Gal{\mathop{\mathrm{Gal}}\nolimits}

\def\Jac{\mathop{\mathrm{Jac}}\nolimits}
\def\Ker{\mathop{\mathrm{Ker}}\nolimits}

\def\GL{\mathop{\mathrm{GL}}\nolimits}

\def\Pic{\mathop{\mathrm{Pic}}\nolimits}

\def\det{\mathop{\mathrm{det}}\nolimits}

\def\divi{\mathop{\mathrm{div}}}

\begin{document}

\title{The arithmetic of a twist of the Fermat quartic}

\author{Yasuhiro Ishitsuka}
\address{
Institute of Mathematics for Industry, Kyushu University,
Fukuoka, 819-0395, Japan}
\email{yishi1093@gmail.com}

\author{Tetsushi Ito}
\address{Department of Mathematics, Faculty of Science, Kyoto University, Kyoto 606-8502, Japan}
\email{tetsushi@math.kyoto-u.ac.jp}

\author{Tatsuya Ohshita}

\address{Department of Mathematics, 
Cooperative Faculty of Education, Gunma University,
Maebashi, Gunma 371-8510, Japan}
\email{ohshita@gunma-u.ac.jp}

\date{\today}

\begin{abstract}
We study the arithmetic of the twist of the Fermat quartic defined by
$X^4 + Y^4 + Z^4 = 0$ which has no $\Q$-rational point.
We calculate the Mordell--Weil group of the Jacobian variety explicilty.
We show that the degree $0$ part of the Picard group is a free $\Z/2\Z$-module of rank $2$,
whereas the Mordell--Weil group is a free $\Z/2\Z$-module of rank $3$.
Thus the relative Brauer group is non-trivial.
We also show that this quartic violates the local-global property for
linear determinantal representations.
\end{abstract}

\maketitle

\section{Introduction}

The aim of this paper is to study the arithmetic of
the smooth plane quartic
\[
C_4 \coloneqq \{\, [X : Y : Z] \in \PP^2 \mid X^4 + Y^4 + Z^4 = 0 \,\}.
\]
It is a twist of the Fermat quartic.
The quartic $C_4$ has no $\Q$-rational point because it has no $\R$-rational point.
We shall explicitly calculate the Picard group $\Pic(C_4)$
and the Mordell--Weil group of the Jacobian variety $\Jac(C_4)$.
Theses results rely on our previous calculation of
the Mordell--Weil group
of the Jacobian variety of the Fermat quartic over $\Q(\zeta_8)$
summarized in \cite{IshitsukaItoOhshita:FermatQuartic}.

Define the divisors $D_i$ ($0 \leq i \leq 3$) on the quartic $C_4$ by
\begin{align*}
D_0 &\coloneqq [1 :  \zeta_3 :  \zeta^2_3] + [1 :  \zeta^2_3 :  \zeta_3] &
D_1 &\coloneqq [1 : -\zeta_3 :  \zeta^2_3] + [1 : -\zeta^2_3 :  \zeta_3] \\
D_2 &\coloneqq [1 :  \zeta_3 : -\zeta^2_3] + [1 :  \zeta^2_3 : -\zeta_3] &
D_3 &\coloneqq [1 : -\zeta_3 : -\zeta^2_3] + [1 : -\zeta^2_3 : -\zeta_3]
\end{align*}
Here we put $\zeta_n \coloneqq \exp(2 \pi \sqrt{-1}/n)$.
The divisors $D_0, D_1, D_2, D_3$ are defined over $\Q$, and
twice of them are divisors cut out by bitangents defined over $\Q$; 
see Section \ref{Section:Bitangents}
for details.
The divisor classes of the differences $[D_i - D_0]$ ($1 \leq i \leq 3$) are elements of $\Pic^0(C_4)$ killed by $2$.
They satisfy the following relation;
see Lemma \ref{Lemma:DivisorClassBitangent}
 (3).
\[
  [D_1 - D_0] + [D_2 - D_0] = [D_3 - D_0].
\]
We also consider the following divisor defined over $\Q(\zeta_8)$
\[
  E \coloneqq 2 [1 :  0 :  \zeta^3_8] - 2 [1 :  0 :  \zeta^7_8].
\]
We shall show that $E$ is not linearly equivalent to a divisor defined over $\Q$,
but the divisor class $[E]$
is invariant under the action of $\Gal(\Q(\zeta_8)/\Q)$.
Thus it gives a $\Q$-rational point of Jacobian variety $\Jac(C_4)$.

The following theorem is the main theorem of this paper.

\begin{thm}
\label{MainTheorem}
\begin{enumerate}
\item
The degree $0$ part $\Pic^0(C_4)$ of the Picard group 
is a free $\Z/2\Z$-module of rank $2$ generated by
$[D_1 - D_0]$ and $[D_2 - D_0]$:
\[
  \Pic^0(C_4) = (\Z/2\Z)[D_1 - D_0] \oplus (\Z/2\Z)[D_2 - D_0].
\]

\item
The Mordell--Weil group $\Jac(C_4)(\Q)$
is a free $\Z/2\Z$-module of rank $3$ generated by
$[D_1 - D_0]$, $[D_2 - D_0]$, and $[E]$:
\[
  \Jac(C_4)(\Q) =
    (\Z/2\Z)[D_1 - D_0] \oplus (\Z/2\Z)[D_2 - D_0]
    \oplus (\Z/2\Z)[E].
\]
\end{enumerate}
\end{thm}

Let $\Pic_{C_4/\Q}$ be the Picard scheme of $C_4$ over $\Q$
which represents the relative Picard functor;
see \cite[Section 8.2]{BoschLuetkebohmertRaynaud}.
The degree $0$ part $\Pic^0_{C_4/\Q}$ is the Jacobian variety $\Jac(C_4)$,
and the degree $d$ part $\Pic^d_{C_4/\Q}$ is a torsor under $\Jac(C_4)$ for any $d \in \Z$.
Since $C_4$ has a divisor of degree $2$ over $\Q$ (such as $ D_i$ for $0 \leq i \leq 3$),
the scheme $\Pic^{2d}_{C_4/\Q}$ has a $\Q$-rational point,
and it is a trivial torsor under $\Jac(C_4)$.

We shall show that
$\Pic^{2d+1}_{C_4/\Q}$ is a non-trivial torsor under $\Jac(C_4)$ for any $d \in \Z$.

\begin{thm}
\label{MainTheorem:NontrivialTorsor}
Let $d \in \Z$ be an integer.
The scheme $\Pic^{2d+1}_{C_4/\Q}$ has no $\Q$-rational point.
In other words, $\Pic^{2d+1}_{C_4/\Q}$ is a non-trivial torsor under $\Jac(C_4)$.
\end{thm}

\begin{rem}
We have the following exact sequence
\[
\xymatrix{
0 \ar^-{}[r] & \Pic^0(C_4) \ar^-{\iota_{C_4}}[r] & \Jac(C_4)(\Q) \ar^-{\Br}[r] & \Br(C_4/\Q). }
\]
Here
\[ \Br(C_4/\Q) \coloneqq \Ker(\Br(\Q) \to \Br(C_4)) \]
is the relative Brauer group;
see \cite[Theorem 2.1]{CiperianiKrashen},
\cite[Section 8.1, Proposition 4]{BoschLuetkebohmertRaynaud}.
Theorem \ref{MainTheorem} shows that
$\Br([E])$ is non-trivial in $\Br(C_4/\Q)$.
Thus $\Br(C_4/\Q)$ is non-trivial, and $\iota_{F}$ is not surjective.
This is a peculiar phenomenon for curves without rational points
because, if a curve has a rational point,
then the relative Brauer group vanishes.
After this work was completed,
Brendan Creutz kindly informed that
he constructed plane quartics over $\Q$
for which
the degree $0$ part of the Picard group is strictly smaller than
the Mordell--Weil group of the Jacobian variety; see \cite[Theorem 6.3, Remark 6.4]{Creutz} for details.
\end{rem}

We shall give an application to quadratic points.
Once $\Pic^2(C_4)$ is calculated,
we can determine all of the rational points on $C_4$
defined over quadratic extensions of $\Q$
by Faddeev's methods \cite{Faddeev}.
It turns out that all of them are defined over $\Q(\sqrt{-3})$
and they are the points of tangency of bitangents defined over $\Q$.
(The following result also follows from
the results in \cite[Section 7]{IshitsukaItoOhshita:FermatQuartic}.)

\begin{thm}
\label{MainTheorem:QuadraticPoints}
There are exactly $8$ rational points on $C_4$
which are defined over quadratic extensions of $\Q$.
These are
\[ [1 : \pm \zeta_3 : \pm \zeta^2_3] \quad \text{and} \quad [1 : \pm \zeta^2_3 : \pm \zeta_3]. \]
\end{thm}

Finally, we shall give an application to
the linear determinantal representations of the quartic $C_4$.
Recall that a \textit{linear determinantal representation} of a quartic $C \subset \PP^2$
over a field $K$ is a triple of 4 by 4 matrices
$M_0, M_1, M_2 \in \mathrm{Mat}_4(K)$
such that
\[ \det(XM_0 + YM_1 + ZM_2) = 0 \]
is a defining equation of the curve $C$.
The problem to find linear determinantal representations
of a given plane curve is motivated by
Arithmetic Invariant Theory of Bhargava--Gross \cite{BhargavaGross}.
In \cite{Ishitsuka:Proportion},
the first author studied the arithmetic properties of
linear determinantal representations of smooth cubics.
In \cite{IshitsukaItoOhshita:LDR},
the authors of the current paper determined all of the linear determinantal representations
of the Fermat quartic and the Klein quartic.

As an application of our explicit calculation of $\Pic^2(C_4)$,
we shall show that $C_4$ violates
the local-global property of linear determinantal representations.

\begin{thm}
\label{MainTheorem:LDR}
\begin{enumerate}
\item The quartic $C_4$ does not admit a linear determinantal representation over $\Q$.
\item For $K = \R$ or a $p$-adic field $\Q_p$,
the quartic $C_4$ admits a linear determinantal representation over $K$.
\end{enumerate}
\end{thm}

Here is a brief sketch of the proof of our results.
In \cite{IshitsukaItoOhshita:FermatQuartic},
based on the results of Faddeev and Rohrlich,
we explicitly calculated the Mordell--Weil group of the \textit{Fermat quartic}
$F_4 \subset \PP^2$ defined by
$X^4 + Y^4 = Z^4$.
Since the quartic $C_4$ is isomorphic to the Fermat quartic $F_4$
over $\Q(\zeta_8)$, we have
\[ \Jac(C_4)(\Q(\zeta_8)) \cong \Jac(F_4)(\Q(\zeta_8))
   \cong (\Z/4\Z)^{\oplus 5} \oplus \Z/2\Z. \]
Calculating the twist of Galois action explicitly,
we calculate $\Jac(C_4)(\Q)$ and $\Pic^0(C_4)$.
Since $\Pic^0(C_4) \cong \Pic^2(C_4)$,
we determine every $\Q$-rational divisor of degree $2$ on $C_4$.
Theorem \ref{MainTheorem:QuadraticPoints}
is obtained by Faddeev's methods.
Theorem \ref{MainTheorem:NontrivialTorsor} and Theorem \ref{MainTheorem:LDR}
are consequences of our explicit calculation.

In the appendix,
we give a sample source code for Singular (version 4.2.1)
which confirms the formulae on divisors used in this paper.

\section{Bitangents and their points of tangency}
\label{Section:Bitangents}

Recall that a line $L \subset \PP^2$ is called
a \textit{bitangent} of $C_4$
if the intersection multiplicity at every $P \in C_4 \cap L$ is divisible by $2$.
The divisor on $C_4$ cut out by $L$ is divisible by $2$.
Since $C_4$ is a smooth quartic, it has $28$ bitangents over $\overline{\Q}$.
The defining equations of them are well-known,
and can be found in \cite[p.\ 14]{Edge}.

Among the 28 bitangents, the following four bitangents are defined over $\Q$:
\begin{align*}
L_0 : X + Y + Z &= 0, &
L_1 : X - Y + Z &= 0, \\
L_2 : X + Y - Z &= 0, &
L_3 : X - Y - Z &= 0.
\end{align*}
They give divisors of degree $2$ over $\Q$:
\begin{align*}
D_0 &\coloneqq [1 :  \zeta_3 :  \zeta^2_3] + [1 :  \zeta^2_3 :  \zeta_3] &
D_1 &\coloneqq [1 : -\zeta_3 :  \zeta^2_3] + [1 : -\zeta^2_3 :  \zeta_3] \\
D_2 &\coloneqq [1 :  \zeta_3 : -\zeta^2_3] + [1 :  \zeta^2_3 : -\zeta_3] &
D_3 &\coloneqq [1 : -\zeta_3 : -\zeta^2_3] + [1 : -\zeta^2_3 : -\zeta_3]
\end{align*}

\begin{lem}
\label{Lemma:DivisorClassBitangent}
\begin{enumerate}
\item For any $i,j$ with $0 \leq i < j \leq 3$,
the divisor $D_i$ is not linearly equivalent to $D_j$.
In particular, $\Pic^2(C_4)$ has at least $4$ elements.

\item For any $i,j$ with $0 \leq i < j \leq 3$,
the divisor class $[D_i - D_j]$ is killed by $2$
in $\Pic^0(C_4)$.

\item 
We have an injective homomorphism
\[
  (\Z/2\Z)^{\oplus 2} \hookrightarrow \Pic^0(C_4),
  \qquad
  (c_1, c_2) \mapsto c_1 [D_1 - D_0] + c_2 [D_2 - D_0].
\]
Moreover, we have
$[D_1 - D_0] + [D_2 - D_0] = [D_3 - D_0]$
in $\Pic^0(C_4)$.
\end{enumerate}
\end{lem}

\begin{proof}
(1)
Assume that $D_i$ is linearly equivalent to $D_j$.
Then there is a non-zero rational function $f$ on $C_4$
with $\divi(f) = D_i - D_j$.
This implies the morphism $C_4 \to \PP^1$ induced by $f$
has degree $1$,
which is absurd because $C_4$ has genus $3$.

(2)
By definition of bitangents,
the divisors on $C_4$ cut out by $L_i, L_j$ are $2 D_i, 2 D_j$.
Thus $2 D_i$ and $2 D_j$ are linearly equivalent to
each other because both are hyperplane sections.
Therefore, we have $[2 D_i - 2 D_j] = 0$ in $\Pic^0(C_4)$.

(3)
Since
\begin{align*}
    \mathrm{div}(X^2+Y^2+Z^2) &= D_0 + D_1 + D_2 + D_3, \\
    \mathrm{div}(X+Y+Z) &= 2D_0,
\end{align*}
we have
\[ D_1 + D_2 + D_3 - 3 D_0 =
   \divi \left(\frac{X^2+Y^2+Z^2}{(X+Y+Z)^2} \right).
\]
Hence we have $[D_1 - D_0] + [D_2 - D_0] = [D_3 - D_0]$ in $\Pic^0(C_4)$.
The injectivity follows from (1).
\end{proof}

\section{Divisors on the Fermat quartic}
\label{Section:DivisorsFermatQuartic}

Let $F_4 \subset \PP^2$ be the Fermat quartic over $\Q$
defined by $X^4 + Y^4 = Z^4$.
Rohrlich calculated the subgroup of $\Jac(F_4)(\Q(\zeta_8))$
generated by the differences of the cusps;
see \cite[p.\ 117, Corollary 1]{Rohrlich}.
(This subgroup is denoted by $\mathscr{D}^{\infty}/\mathscr{F}^{\infty}$ in \cite{Rohrlich}.)
In \cite{IshitsukaItoOhshita:FermatQuartic},
with the aid of computers,
we proved there are no other elements in $\Jac(F_4)(\Q(\zeta_8))$.
Thus, we have an isomorphism
\[ \Jac(F_4)(\Q(\zeta_8)) \cong (\Z/4\Z)^{\oplus 5} \oplus (\Z/2\Z). \]
Explicit generators are given in
\cite[Theorem 6.5]{IshitsukaItoOhshita:FermatQuartic}.

The quartic $C_4$ is isomorphic to
the Fermat quartic $F_4$ over $\Q(\zeta_8)$.
We fix an isomorphism between $C_4$ and $F_4$ over $\Q(\zeta_8)$ as follows:
\[
  \rho \colon C_4 \overset{\cong}{\to} F_4, \qquad
  [X: Y: Z] \mapsto [X: Y: \zeta_8 Z].
\]

Via the above isomorphism $\rho$,
we shall translate the results in \cite{IshitsukaItoOhshita:FermatQuartic}
into the results on divisor classes on $C_4$.
We define $A_i, B_i, C_i \in C_4$ ($0 \leq i \leq 3$) by
\begin{align*}
A_i &\coloneqq [0 : \zeta_4^i : \zeta_8^7], &
B_i &\coloneqq [\zeta_4^i : 0 : \zeta_8^7], &
C_i &\coloneqq [\zeta_8 \zeta_4^i : 1 : 0].
\end{align*}
There points correspond to the cusps on $F_4$ via
the isomorphism $\rho$.
Then we define
\begin{align*}
\alpha_{i} &\coloneqq [A_{i} - B_{0}], &
\beta_{i}  &\coloneqq [B_{i} - B_{0}], &
\gamma_{i} &\coloneqq [C_{i} - B_{0}].
\end{align*}
Here the linear equivalence class of a divisor $D$ is denoted by $[D]$.
Moreover, we define
\begin{align*}
e_{1} &\coloneqq \alpha_{1}, &
e_{2} &\coloneqq \alpha_{2}, &
e_{3} &\coloneqq \beta_{1}, \\
e_{4} &\coloneqq \beta_{2}, &
e_{5} &\coloneqq \gamma_{1}, 
\end{align*}
and
\[
e_{6} \coloneqq
\alpha_{1} + \alpha_{2} + \beta_{1} + \beta_{2} + \gamma_{1} + \gamma_{2}.
\]
Note that we use the same notation as in
\cite[Section 2]{IshitsukaItoOhshita:FermatQuartic}.
The points $A_i, B_i, C_i$ ($0 \leq i \leq 3$)
and the divisors
$\alpha_i, \beta_i, \gamma_i$  ($0 \leq i \leq 3$),
$e_i$ ($1 \leq i \leq 6$)
correspond to the points and the divisors denoted
by the same letters in
\cite[Section 2]{IshitsukaItoOhshita:FermatQuartic}.

Translating \cite[Theorem 6.5]{IshitsukaItoOhshita:FermatQuartic} via $\rho$,
we have an isomorphism
\[
(\Z/4\Z)^{\oplus 5} \oplus (\Z/2\Z) \overset{\cong}{\to}
\Jac(C_4)(\Q(\zeta_8)),
\qquad
(c_1,c_2,c_3,c_4,c_5,c_6) \mapsto \sum_{i=1}^{6} c_i e_{i}
\]

For later use, we write
$\alpha_{i}, \beta_{i}$ and $\gamma_{i}$ ($0 \leq i \leq 3$)
as linear combinations of $e_{i}$  ($1 \leq i \leq 6$).
\begin{align*}
\alpha_{0} &= 2e_{1} + e_{2} + 2e_{3} + e_{4}, &
\alpha_{1} &= e_{1}, \\
\alpha_{2} &= e_{2}, &
\alpha_{3} &= e_{1} + 2e_{2} + 2e_{3} + 3 e_{4}, \\
\beta_{0} &= 0, &
\beta_{1} &= e_{3}, \\
\beta_{2} &= e_{4}, &
\beta_{3} &= 3 e_{3} + 3 e_{4}, \\
\gamma_{0} &= 3 e_{1} + 3 e_{2} + e_{3} + e_{5} + e_{6}, &
\gamma_{1} &= e_{5}, \\
\gamma_{2} &= 3 e_{1} + 3 e_{2} + 3 e_{3} + 3 e_{4} + 3 e_{5} + e_{6}, &
\gamma_{3} &= 2 e_{1} + 2 e_{2} + e_{4} + 3 e_{5}.
\end{align*}

\begin{rem}
\label{Remark:MordellWeilGroupBasis}
It can be checked by Singular that
we have the following relations between divisors.
\begin{align*}
    D_1 - D_0 &= 2 B_1 + 2 B_2 - 4 B_0 \\
    & + \divi \bigg( \frac{(X - \zeta_8Z)(X-Y+Z)}{(X^2+Y^2+Z^2) + (\zeta_{24}^5 - \zeta_{24}^3 - \zeta_{24})(Y^2 - XZ)} \bigg), \\
    D_2 - D_0 &= 2 A_1 + 2 A_2 + 2 B_1 + 2 B_2 - 8 B_0 \\
    & + \divi \bigg( \frac{(X - \zeta_8Z)^2(X+Y-Z)}{(X^2+Y^2+Z^2)(X+Y) - (\zeta_{24}^5 - \zeta_{24}^3 - \zeta_{24})Z(X^2 + XY + Y^2)} \bigg), \\
    D_3 - D_0 &= 2 A_1 + 2 A_2 - 4 B_0 \\
    & + \divi \bigg( \frac{(X - \zeta_8Z)(X-Y-Z)}{(X^2-Y^2-2YZ-Z^2) + (\zeta_{24}^5 - \zeta_{24}^3 - \zeta_{24})(Y^2 + YZ + Z^2)} \bigg), \\
    E &= 2 B_2 - 2 B_0.
\end{align*}
Therefore, we have the following relations in $\Jac(C_4)(\Q(\zeta_8))$:
\begin{align*}
[D_1 - D_0] &=  2 e_3 + 2 e_4, &
[D_2 - D_0] &=  2 e_1 + 2 e_2 + 2 e_3 + 2 e_4, \\
[D_3 - D_0] &=  2 e_1 + 2 e_2, & 
[E] &= 2 e_4.
\end{align*}

\end{rem}

\section{Calculation of the Galois action}

Here we shall calculate the action of Galois group on
$\Jac(C_4)(\Q(\zeta_8))$.
Since
\[
\Jac(C_4)(\Q) = \Jac(C_4)(\Q(\zeta_8))^{\Gal(\Q(\zeta_8)/\Q)},
\]
we can calculate $\Jac(C_4)(\Q)$ explicitly
and prove Theorem \ref{MainTheorem} (1).

The Galois group $\Gal(\Q(\zeta_8)/\Q) \cong (\Z/8\Z)^{\times}$
is generated by $\sigma_3, \sigma_5$ defined by
\[
\sigma_3 \colon \zeta_8 \mapsto \zeta_8^3, \qquad
\sigma_5 \colon \zeta_8 \mapsto \zeta_8^5.
\]

\begin{prop}
\label{Proposition:MordellWeilGroupOverQ}
The Mordell--Weil group $\Jac(C_4)(\Q)$
is a free $\Z/2\Z$-module of rank $3$ generated by
$[D_1 - D_0]$, $[D_2 - D_0]$, $[E]$.
\end{prop}

\begin{proof}
The actions of $\sigma_3$, $\sigma_5$ on
$A_i, B_i, C_i$ ($0 \leq i \leq 3$)
are summarized in the following table:
\[
\begin{array}{|c||c|c|c|c||c|c|c|c||c|c|c|c|}
\hline
& A_0 & A_1 & A_2 & A_3 & B_0 & B_1 & B_2 & B_3 & C_0 & C_1 & C_2 & C_3 \\
\hline
\sigma_3 & A_1 & A_0 & A_3 & A_2 & B_1 & B_0 & B_3 & B_2 & C_1 & C_0 & C_3 & C_2 \\
\hline
\sigma_5   & A_2 & A_3 & A_0 & A_1 & B_2 & B_3 & B_0 & B_1 & C_2 & C_3 & C_0 & C_1 \\
\hline
\end{array}
\]

Using the formulae in
Section \ref{Section:DivisorsFermatQuartic}
(see also \cite[Section 3]{IshitsukaItoOhshita:FermatQuartic}),
the actions of $\sigma_3$, $\sigma_5$ on
the basis $e_i$ ($1 \leq i \leq 6$) of $\Jac(C_4)(\Q(\zeta_8))$
are calculated as follows.
\begin{align*}
&\begin{array}{|c||c|c|c|c|}
\hline
& e_1 & e_2 & e_3 & e_4 \\
\hline
\sigma_3 & 2e_1 + e_2 + e_3 + e_4 & e_1+ 2e_2 + e_3 + 3e_4  & 3e_3 & 2e_3+3e_4 \\
\hline
\sigma_5   & e_1 + 2e_2 + 2e_3 + 2e_4 & 2e_1 + e_2 + 2e_3 & 3e_3 + 2e_4 & 3e_4 \\
\hline
\end{array} \\
&\begin{array}{|c||c|c|}
\hline
& e_5 & e_6 \\
\hline
\sigma_3 & 3e_1 + 3e_2 + e_5+e_6 & 2e_3 + e_6 \\
\hline
\sigma_5 & 2e_1 + 2e_2 + 3e_5 & 2e_4 + e_6 \\
\hline
\end{array}
\end{align*}
In other words, we have the following matrix representation of $\sigma_3, \sigma_5$ on $\Jac(C_4)(\Q(\zeta_8))$:
\[
	s_3 = \begin{pmatrix}
	2 & 1 & 0 & 0 & 3 & 0 \\
	1 & 2 & 0 & 0 & 3 & 0 \\
	1 & 1 & 3 & 2 & 0 & 2 \\
	1 & 3 & 0 & 3 & 0 & 0 \\
	0 & 0 & 0 & 0 & 1 & 0 \\
	0 & 0 & 0 & 0 & 1 & 1
	\end{pmatrix}, \qquad
	s_5 =  \begin{pmatrix}
	1 & 2 & 0 & 0 & 2 & 0 \\
	2 & 1 & 0 & 0 & 2 & 0 \\
	2 & 2 & 3 & 0 & 0 & 0 \\
	2 & 0 & 2 & 3 & 0 & 2 \\
	0 & 0 & 0 & 0 & 3 & 0 \\
	0 & 0 & 0 & 0 & 0 & 1 
	\end{pmatrix}.
\]

The $\Q$-rational points of $\Jac(C_4)$ are
the fixed points by $\sigma_3, \sigma_5$.
We have to study the kernel of $s_3 - 1$ and $s_5 - 1$.
It is straightforward to check that
the kernel is a free $\Z/2\Z$-module of rank $3$
generated by $2 e_1 + 2 e_2$, $2 e_3$, $2 e_4$.
Thus, the assertion follows from
Remark \ref{Remark:MordellWeilGroupBasis}.
\end{proof}

Next, we shall calculate the action of Galois group on
$\Pic^1(C_4)(\Q(\zeta_8))$.
Although $\Pic^1(C_4)(\Q(\zeta_8))$ is isomorphic to $\Jac(C_4)(\Q(\zeta_8))$,
the Galois actions are different.

\begin{prop}
\label{Proposition:Pic1}
\begin{enumerate}
\item Each automorphism $\sigma_3, \sigma_5, \sigma_3 \sigma_5 \in \Gal(\Q(\zeta_8)/\Q)$
does not have a fixed point in $\Pic^1_{C_4/\Q}(\Q(\zeta_8))$.
\item For $K = \Q(\sqrt{-1}), \Q(\sqrt{2})$ and $\Q(\sqrt{-2})$, the scheme $\Pic^1_{C_4/\Q}$ has no $K$-rational point.
\end{enumerate}
\end{prop}

\begin{proof}
(1)
By the matrix representations $s_3, s_5$ of the automorphisms 
$\sigma_3, \sigma_5 \in \Gal(\Q(\zeta_8)/\Q)$ 
in the proof of Proposition \ref{Proposition:MordellWeilGroupOverQ},
we see that 
\begin{equation}\label{Equation:image}
    (\sigma_5-1)\Jac(C_4)(\Q(\zeta_8)) = 2 \Jac(C_4)(\Q(\zeta_8)).
\end{equation}
Moreover, if 
\(
a_1e_1 + a_2e_2 + a_3e_3 + a_4e_4 + a_5e_5 + a_6e_6 \; (a_i \in \Z)
\)
is in \((\sigma_3-1)\Jac(C_4)(\Q(\zeta_8))\), then we have 
\begin{equation}\label{Equation:equivalence}
    a_2 + a_3 + a_6 \equiv 0 \pmod 2.
\end{equation}
Since $s_3s_5 - 1 = s_3(s_5-1) + s_3-1 \equiv s_3 - 1 \pmod 2$, we have the same congruence on \((\sigma_3\sigma_5-1)\Jac(C_4)(\Q(\zeta_8))\).

Since $A_0 = [0 : 1 : \zeta_8^7]$ is a $\Q(\zeta_8)$-rational point, we have the following bijection over $\Q(\zeta_8)$
\[
\Jac(C_4)(\Q(\zeta_8)) \overset{1:1}{\to}
\Pic^1_{C_4/\Q}(\Q(\zeta_8)), \qquad
[D] \mapsto [D + A_0].
\]
If a divisor class $[D + A_0]$ is fixed by $\sigma_5$,
we have
$\sigma_5([D + A_0]) = [D + A_0]$.
In particular, we have
$(\sigma_5 - 1)[A_0] = - (\sigma_5 - 1)[D]$.
Using the formulae in
Section \ref{Section:DivisorsFermatQuartic}
(see also \cite[Section 3]{IshitsukaItoOhshita:FermatQuartic}),
we see that
\[
  (\sigma_5 - 1)[A_0] = [A_2 - A_0] = \alpha_2 - \alpha_0 = 2 e_1 + 2 e_3 + 3 e_4
\]
is not divisible by $2$ in $\Jac(C_4)(\Q(\zeta_8))$.
On the other hand, $-(\sigma_5-1)[D]$ is an element of 
$2 \Jac(C_4)(\Q(\zeta_8))$ by \eqref{Equation:image}.
The contradiction shows
$\sigma_5$ does not have a fixed point in $\Pic^1_{C_4/\Q}(\Q(\zeta_8))$.
Since two divisor classes
\begin{align*}
    (\sigma_3 - 1)[A_0] &= [A_1 - A_0] = \alpha_1 - \alpha_0 = 3e_1 + 3e_2 + 2e_3 + 3e_4, \\
    (\sigma_3\sigma_5 - 1)[A_0] &= [A_3 - A_0] = \alpha_3 - \alpha_0 = 3e_1 + e_2 + 2e_4
\end{align*}
do not satisfy the congruence \eqref{Equation:equivalence},
by similar arguments,
we see that neither $\sigma_3$ nor $\sigma_3\sigma_5$
has a fixed point in $\Pic^1_{C_4/\Q}(\Q(\zeta_8))$.

(2) The assertion for $\sigma_5$ follows from
\[ \Q(\zeta_8)^{\sigma_5} \coloneqq \{ x \in \Q(\zeta_8) \mid \sigma_5(x) = x \} = \Q(\sqrt{-1}). \]
Other parts follow from $\Q(\zeta_8)^{\sigma_3} = \Q(\sqrt{-2})$ and 
$\Q(\zeta_8)^{\sigma_3\sigma_5} = \Q(\sqrt{2})$.
\end{proof}

\begin{rem}
The class
$[\Pic^1_{C_4}/\Q] \in H^1(\Gal(\overline{\Q}/\Q), \Jac(C_4)(\overline{\Q}))$
is a non-trivial element killed by $2$.
The class
$[\Pic^2_{C_4}/\Q] \in H^1(\Gal(\overline{\Q}/\Q), \Jac(C_4)(\overline{\Q}))$
is trivial because $C_4$ has a divisor of degree $2$ over $\Q$.
In fact, we have
$2 [\Pic^1_{C_4/\Q}] = [\Pic^2_{C_4/\Q}] = 0$
in $H^1(\Gal(\overline{\Q}/\Q), \Jac(C_4)(\overline{\Q}))$.
\end{rem}

Finally,
we shall calculate the action of Galois group on the divisor 
\[ E \coloneqq 2 [1 :  0 :  \zeta^3_8] - 2 [1 :  0 :  \zeta^7_8]. \]

\begin{prop}
\label{Proposition:BrauerObstruction}
\begin{enumerate}
\item The divisor class $[E] \in \Jac(C_4)(\Q(\zeta_8))$ is invariant under the action of $\Gal(\Q(\zeta_8)/\Q)$.
\item The divisor $E$ is not linearly equivalent to a divisor defined over $\Q$.
In other words, the image of $[E]$ in the relative Brauer group
$\Br(C_4/\Q)$ is non-trivial.
\end{enumerate}
\end{prop}

\begin{proof}
(1) Since $[E] = 2 e_4$, we have
\[
  \sigma_3([E]) = 2 \sigma_3(e_4) = 2 e_4 = [E], \quad
  \sigma_5([E]) = 2 \sigma_5(e_4) = 2 e_4 = [E].
\]
(See Remark \ref{Remark:MordellWeilGroupBasis}
and the table in the proof of Proposition \ref{Proposition:Pic1}.)
Hence $[E]$ is invariant under the action of $\Gal(\Q(\zeta_8)/\Q)$.

(2) 
By (1), we see that the divisors
$E - \sigma_3(E)$, $E - \sigma_5(E)$, $2 [E]$
are linearly equivalent to $0$.
To calculate the obstruction in the relative Brauer group explicitly,
we shall give the rational functions which give linear equivalences.
We have
	\begin{align*}
		\sigma_3(E) &=  2 [1:0:\zeta_8] - 2 [1:0:\zeta_8^5], \\
		\sigma_5(E) &=  2 [1:0:\zeta_8^7] - 2 [1:0:\zeta_8^3] = -E, \\
		\sigma_3\sigma_5(E) &=  2 [1:0:\zeta_8^5] - 2 [1:0:\zeta_8] = -\sigma_3(E).
	\end{align*}
Hence we have
	\begin{align*}
		2E &= \mathrm{div}\left( \frac{X - \zeta_8^5 Z}{X - \zeta_8Z} \right), \\ 
		E + \sigma_3(E) &= \mathrm{div}\left( \frac{Y^2}{(X - \zeta_8 Z)(X - \zeta_8^3 Z)} \right), \\ 
		E - \sigma_3(E) &= \mathrm{div}\left(\frac{Y^2}{(X - \zeta_8 Z)(X - \zeta_8^7Z)} \right).
	\end{align*}

To prove (2), it is enough to show that the image of $\Br([E])$ in
\[ \Br(\R) = H^2(\Gal(\C/\R), \C^{\times}) \cong \{ \pm 1 \} \]
is non-trivial.
We shall calculate the cocycle explicitly.
Let $\tau$ denote the complex conjugation.
We have $\tau = \sigma_3 \sigma_5$ in $\Gal(\Q(\zeta_8)/\Q)$.
We have
$\tau(E) = \sigma_3\sigma_5(E) = -\sigma_3(E)$.
	Let $u_1 = 1$ and 
	\[
	u_\tau = \frac{Y^2}{(X - \zeta_8 Z)(X - \zeta_8^3 Z)}.
	\]
	We have $\mathrm{div}(u_{\tau}) = E + \sigma_3(E) = E - \tau(E)$.
	Thus a 2-cocycle
	$a \colon \{1, \tau\}^2 \to \{ \pm 1 \}$
	corresponding to $\Br([E])$ is given by
	\[
		a_{s, t} = \frac{u_{s} \cdot s(u_{t})}{u_{st}}
		\qquad
		(s,t \in \Gal(\C/\R) = \{ 1, \tau \}).
	\]
	Since $u_1 = 1$, we have $a_{1,1} = a_{1,\tau} = a_{\tau, 1} = 1$.
	When $s = t = \tau$, we have
	\begin{align*}
		a_{\tau, \tau} &= \frac{u_{\tau} \cdot \tau(u_{\tau})}{u_{\tau^2}} = u_{\tau} \cdot \tau(u_{\tau}) \\
		&= \frac{Y^2}{(X - \zeta_8 Z)(X - \zeta_8^3 Z)} \cdot \frac{Y^2}{(X - \zeta_8^7 Z)(X - \zeta_8^5 Z)}
		= \frac{Y^4}{X^4 + Z^4} \\
		&= -1.
	\end{align*}
	Therefore, the cocycle $a$ gives a non-trivial element of
	$\Br(\R) \cong \{ \pm 1 \}$
	corresponding to Hamilton's quarternion algebra $\mathbb{H}$.
\end{proof}

\section{Proof of Theorem \ref{MainTheorem} and Theorem \ref{MainTheorem:NontrivialTorsor}}

\subsection*{Proof of Theorem \ref{MainTheorem}}

We shall complete the proof of Theorem \ref{MainTheorem}.
We consider the following exact sequence
\[
\xymatrix{
0 \ar^-{}[r] & \Pic^0(C_4) \ar[r] & \Jac(C_4)(\Q) \ar^-{}[r] & \Br(C_4/\Q). }
\]

We have $|\Pic^0(C_4)| \geq 4$
by Lemma \ref{Lemma:DivisorClassBitangent} (3).

On the other hand,
we have $|\Jac(C_4)(\Q)| = 8$
by Proposition \ref{Proposition:MordellWeilGroupOverQ}.
(See also Remark \ref{Remark:MordellWeilGroupBasis}.)
Since $[E] \in \Jac(C_4)(\Q)$ has a non-trivial image in $\Br(C_4/\Q)$
by Proposition \ref{Proposition:BrauerObstruction} (2),
we have
$|\Pic^0(C_4)| < |\Jac(C_4)(\Q)| = 8$.

Theorem \ref{MainTheorem} follows from these results.

\subsection*{Proof of Theorem \ref{MainTheorem:NontrivialTorsor}}

Since the quartic $C_4$ has a $\Q$-rational divisor of degree $2$
(such as $D_i$ ($0 \leq i \leq 3$) in Section \ref{Section:Bitangents}),
we have an isomorphism
$\Pic^{2d+1}_{C_4/\Q} \cong \Pic^{1}_{C_4/\Q}$ over $\Q$.
Hence it is enough to show 
$\Pic^{1}_{C_4/\Q}$ has no $\Q$-rational point.
This follows from Proposition \ref{Proposition:Pic1} (2).

\begin{rem}
Our calculation shows that,
for $K = \Q(\sqrt{-1}), \Q(\sqrt{2})$ and $\Q(\sqrt{-2})$,
$\Pic^{1}_{C_4/K}$ is a non-trivial torsor
under $\Jac(C_4 \otimes_{\Q} K)$.
\end{rem}

\section{Divisors of degree $2$ and quadratic points}

Recall that we have divisors $D_0, D_1, D_2, D_3$
of degree $2$ over $\Q$;
see Section \ref{Section:Bitangents}.
We shall show that these are the only divisors of degree $2$
on $C_4$ over $\Q$.

\begin{prop}
\label{Proposition:PicardGroup2}
The degree $2$ part $\Pic^2(C_4)$ of the Picard group
consists of the four elements $[D_i]$ ($0 \leq i \leq 3$),
i.e.,
\[ \Pic^2(C_4) = \{\, [D_0],\ [D_1],\ [D_2],\ [D_3] \,\} \]
In particular, every $\Q$-rational divisor of degree $2$ on $C_4$
is linearly equivalent to an effective divisor.
\end{prop}

\begin{proof}
We have a bijection
\[ \Pic^2(C_4) \overset{\cong}{\to} \Pic^0(C_4),
   \qquad [D] \mapsto [D - D_0]. \]
Hence
$|\Pic^2(C_4)| = |\Pic^0(C_4)| = 4$
by Theorem \ref{MainTheorem} (1).
It remains to show that
the four divisor classes $[D_i]$ ($0 \leq i \leq 3$)
are distinct from each other.
This follows from Lemma \ref{Lemma:DivisorClassBitangent} (1).
\end{proof}

\subsection*{Proof of Theorem \ref{MainTheorem:QuadraticPoints}}

Since $\Pic^2(C_4)$ is explicitly calculated,
we can determine points on $C_4$ defined over a quadratic extension of $\Q$
by Faddeev's method \cite{Faddeev}.
(See \cite[Lemma 7.1]{IshitsukaItoOhshita:FermatQuartic} for a precise statement.)

Since the proof is essentially same as
\cite[Section 7]{IshitsukaItoOhshita:FermatQuartic},
we briefly give an outline of the proof.
Since $C_4$ has no $\R$-rational point, it has no $\Q$-rational point.
Let $P \in C_4$ be a point defined over a quadratic extension of $\Q$,
and $\overline{P}$ be the conjugate of $P$.
By Proposition \ref{Proposition:PicardGroup2},
the effective divisor $P + \overline{P}$ is linearly equivalent to $D_i$
for some $0 \leq i \leq 3$.
To prove Theorem \ref{MainTheorem:QuadraticPoints},
it is enough to show $P + \overline{P}$ is equal to $D_i$ as divisors.
If $P + \overline{P} \neq D_i$,
there is a non-zero rational function $f$ on $C_4$
with $\divi(f) = P + \overline{P} - D_i$.
The morphism $C \to \PP^1$ induced by $f$ has degree $2$.
This is absurd since $C_4$ is non-hyperelliptic.

\section{Linear determinantal representations}

In this final section, we give an application to
the local--global property of linear determinantal representations.

Recall that
a \textit{linear determinantal representation} of a plane curve
$C \subset \PP^2$ over $K$ of degree $d$ is a triple of $d$ by $d$ matrices
$M = (M_0, M_1, M_2)$ ($M_0, M_1, M_2 \in \mathrm{Mat}_d(K)$)
such that
\[ \det(XM_0 + YM_1 + ZM_2) = 0 \]
is a defining equation of the curve $C \subset \PP^2$.
Two linear determinantal representations
$M = (M_0, M_1, M_2)$, $M' = (M'_0, M'_1, M'_2)$
are said to be \textit{equivalent} if
there are invertible matrices $P, Q \in \GL_d(K)$
such that $M'_i = P M_i Q$ for $i = 0,1,2$.
It is well-known that
there is a natural bijection between the following sets:
\begin{itemize}
\item the set of equivalence classes of linear determinantal representations of $C$ over $K$, and
\item the set of linear equivalence classes of $K$-rational divisors $D$ on $C$ of degree $g(C) - 1 = d(d-3)/2$
which are not linearly equivalent to effective divisors.
(Here $g(C) = (d-1)(d-2)/2$ is the genus of $C$.)
\end{itemize}
(See \cite[Theorem 1]{IshitsukaItoOhshita:LDR}. See also \cite{Beauville}, \cite{Dolgachev}.)

\subsection*{Proof of Theorem \ref{MainTheorem:LDR}}

To prove Theorem \ref{MainTheorem:LDR} (1),
it is enough to prove that every divisor of degree $2$
on $C_4$ over $\Q$ is linearly equivalent to an effective divisor.
This follows from Proposition \ref{Proposition:PicardGroup2}.

Theorem \ref{MainTheorem:LDR} (2) follows from the following
general results on linear determinantal representations over large fields.
Recall the a field $K$ is called a \textit{large field}
if any smooth algebraic curve over $K$ with a $K$-rational point
has infinitely many $K$-rational points \cite{Pop}.
(Large fields are also called \textit{ample fields} in \cite{Jarden}.)
Note that $\R$ and $p$-adic fields $\Q_p$ are large fields.

\begin{prop}
\label{Proposition:LargeFieldLDR}
Let $K$ be a large field.
Let $C \subset \PP^2$ be a smooth plane curve of degree $d \geq 3$
defined over $K$.
Assume that there exists a $K$-rational divisor of
degree $g(C) - 1 = d(d-3)/2$ on $C$.
Then $C$ admits infinitely many equivalence classes of
linear determinantal representations over $K$.
\end{prop}

\begin{proof}
For smooth cubics, this result is proved by the first author in
\cite[Theorem 8.2]{Ishitsuka:Proportion}.
Essentially the same proof applies to smooth curves of higher degree.

There exists an exact sequence
\[
\xymatrix{
0 \ar^-{}[r] & \Pic(C) \ar[r] & \Pic_{C/K}(K) \ar^-{\Br}[r] & \Br(C/K) \ar[r] & 0. }
\]
(See \cite[Theorem 2.1]{CiperianiKrashen},
\cite[Section 8.1, Proposition 4]{BoschLuetkebohmertRaynaud}.)
Take a finite separable extension $L/K$ such that $\Pic^{g(C) - 1}_{C/K}$ has an $L$-rational point.
We put $n \coloneqq [L : K]$.
Then the relative Brauer group $\Br(C/K)$ is killed by $n$ by
\cite[Lemma 4.2]{Ishitsuka:Proportion}.

Let $D_0$ be a $K$-rational divisor of degree $g(C) - 1$ on $C$.
Since $[D_0] \in \Pic^{g(C) - 1}_{C/K}$ comes from a $K$-rational divisor,
we have $\Br([D_0]) = 0$.

We have an isomorphism
$\Pic^{g(C) - 1}_{C/K} \cong \Jac(C)$ sending $[D]$ to $[D] - [D_0]$.
Let
\[ W \subset \Pic^{g(C) - 1}_{C/K} \]
be the closed subscheme
whose geometric points corresponds to divisors of degree $g(C) - 1$
which are linearly equivalent to effective divisors.
Then we have $\dim W = g(C) - 1$ because $W$ coincides with the image of the map
\[ \underbrace{C \times \cdots \times C}_{\text{$(g(C) - 1)$ times}}  \to \Pic^{g(C) - 1}_{C/K}. \]
Let $[n] \colon \Jac(C) \to \Jac(C)$ be
the multiplication-by-$n$ isogeny
and consider the subvariety $Z \subset \Jac(C)$ defined by
\[ Z \coloneqq [n]^{-1}(W - [D_0]). \]
It is a subvariety of $\Jac(C)$ of dimension $g(C) - 1$ defined over $K$.

Since $K$ is a large field, there exist infinitely many
$K$-rational points $P \in \Jac(C)(K)$ outside $Z$.
(To see this, take a smooth curve $C' \subset \Jac(C)$ passing through $0$
with $C' \not \subset Z$.
Since $K$ is a large field, the curve $C'$ contains infinitely many
$K$-rational points.)
Since $[n](P) \notin W - [D_0]$,
we see that $[n](P) + [D_0]$ is a $K$-rational point of $\Pic^{g(C) - 1}_{C/K}$
which does not sit in $W$.
Since
\[
  \Br([n](P) + [D_0]) = n \Br(P) + \Br([D_0]) = 0,
\]
the $K$-rational point $[n](P) + [D_0] \in \Pic^{g(C) - 1}_{C/K}(K)$
comes from a $K$-rational divisor on $C$
which is not linearly equivalent to an effective divisor.

Therefore, the plane curve $C \subset \PP^2$ admits
infinitely many equivalence classes of linear determinantal representations
over $K$.
\end{proof}

The quartic $C_4$ has a $\Q$-rational divisor of degree $2$
such as $D_i$ ($0 \leq i \leq 3$) in Section  \ref{Section:Bitangents}.
By Proposition \ref{Proposition:LargeFieldLDR},
we conclude that $C_4$ admits linear determinantal representations
over $\R$ and $p$-adic fields $\Q_p$
The proof of Theorem \ref{MainTheorem:LDR} is complete.

\begin{rem}
A linear determinantal representation $M = (M_0, M_1, M_2)$
is called \textit{symmetric} if the matrices $M_0, M_1, M_2$ are symmetric matrices.
We may ask similar questions for the existence of
\textit{symmetric determinantal representations}.
Symmetric determinantal representations are classical objects
related to bitangents and theta characteristics
studied by algebraic geometers in the 19th century.
(For historical account, see \cite{Beauville}, \cite{Dolgachev}.)
The local--global property of symmetric determinantal
representations is studied in 
\cite{IshitsukaIto:SDR}, \cite{IshitsukaIto:SDRChar2}, \cite{IIOTU}.
Recently, we constructed a smooth plane quartic over $\Q$ which violates
the local--global property of symmetric determinantal representations; see \cite[Section 7]{IIOTU} for details.
\end{rem}

\appendix

\section{Calculation of divisors}

Here is a sample source code for Singular (version 4.2.1)
which confirms the formulae on divisors used in this paper.

\begin{lstlisting}[basicstyle=\tiny\ttfamily, numbers=left, frame=single, language=C++]
// Calculation of divisors on the quartic
// X^4 + Y^4 + Z^4 = 0 over Q(zeta_8)
// d = zeta_{24}
// minimal polynomial of d = X^8 - X^4 + 1
//
// Singular (version 4.2.1)

LIB "divisors.lib";
ring r=(0,d),(x,y,z),dp; minpoly = d^8 - d^4 + 1;
ideal I = x^4 + y^4 + z^4; qring Q = std(I);
number z8 = d^3; number z4 = z8^2; number z3 = d^8;
divisor A0 = makeDivisor(ideal(x, z8^7 * y - z4^0 * z), ideal(1));
divisor A1 = makeDivisor(ideal(x, z8^7 * y - z4^1 * z), ideal(1));
divisor A2 = makeDivisor(ideal(x, z8^7 * y - z4^2 * z), ideal(1));
divisor A3 = makeDivisor(ideal(x, z8^7 * y - z4^3 * z), ideal(1));
divisor B0 = makeDivisor(ideal(y, z8^7 * x - z4^0 * z), ideal(1));
divisor B1 = makeDivisor(ideal(y, z8^7 * x - z4^1 * z), ideal(1));
divisor B2 = makeDivisor(ideal(y, z8^7 * x - z4^2 * z), ideal(1));
divisor B3 = makeDivisor(ideal(y, z8^7 * x - z4^3 * z), ideal(1));
divisor C0 = makeDivisor(ideal(z, x - z8 * z4^0 * y), ideal(1));
divisor C1 = makeDivisor(ideal(z, x - z8 * z4^1 * y), ideal(1));
divisor C2 = makeDivisor(ideal(z, x - z8 * z4^2 * y), ideal(1));
divisor C3 = makeDivisor(ideal(z, x - z8 * z4^3 * y), ideal(1));
divisor D0 = makeDivisor(ideal(z3   * x - y, z3^2 * x - z), ideal(1))
           + makeDivisor(ideal(z3^2 * x - y, z3   * x - z), ideal(1));
divisor D1 = makeDivisor(ideal(-z3   * x - y, z3^2 * x - z), ideal(1))
           + makeDivisor(ideal(-z3^2 * x - y, z3   * x - z), ideal(1));
divisor D2 = makeDivisor(ideal(z3   * x - y, -z3^2 * x - z), ideal(1))
           + makeDivisor(ideal(z3^2 * x - y, -z3   * x - z), ideal(1));
divisor D3 = makeDivisor(ideal(-z3   * x - y, -z3^2 * x - z), ideal(1))
           + makeDivisor(ideal(-z3^2 * x - y, -z3   * x - z), ideal(1));
divisor E1 = A1 + multdivisor(-1, B0);
divisor E2 = A2 + multdivisor(-1, B0);
divisor E3 = B1 + multdivisor(-1, B0);
divisor E4 = B2 + multdivisor(-1, B0);
divisor E5 = C1 + multdivisor(-1, B0);
divisor E6 = A1 + B1 + C1 + A2 + B2 + C2 + multdivisor(-6, B0);
divisor Lin = multdivisor(4, makeDivisor(ideal(x-z, y), ideal(1)));
divisor E =
    multdivisor( 2, makeDivisor(ideal(z8^3 * x - z, y), ideal(1)))
  + multdivisor(-2, makeDivisor(ideal(z8^7 * x - z, y), ideal(1)));

proc lineqcheck(divisor D1,divisor D2){
  if(linearlyEquivalent(D1,D2)[1]!=0){ return(1);}
  else{ return(0);};};
proc divcheck(string S,divisor D1,divisor D2){
  if(lineqcheck(D1, D2)==1){printf("%s : OK", S);}
  else{printf("%s : x", S);};};
proc prdivcheck(string S, divisor D, poly F1, poly F2){
  if(isEqualDivisor(D, makeDivisor(ideal(F1), ideal(F2)))==1)
  {printf("%s : OK", S);}
  else{printf("%s : x", S);};};

/// Check linear equivalences of divisors
print("");
print("Check linear equivalences of divisors");
divcheck("alpha_0 = 2e_1 + e_2 + 2e_3 + e_4",
  A0 + multdivisor(-1,B0), E1+E1 + E2 + E3+E3 + E4);
divcheck("alpha_1 = e_1",
  A1 + multdivisor(-1,B0), E1);
divcheck("alpha_2 = e_2",
  A2 + multdivisor(-1,B0), E2);
divcheck("alpha_3 = e_1 + 2e_2 + 2e_3 + 4e_4",
  A3 + multdivisor(-1,B0), E1 + E2+E2 + E3+E3 + E4+E4+E4);
divcheck("beta_1 = e_3",
  B1 + multdivisor(-1,B0), E3);
divcheck("beta_2 = e_4",
  B2 + multdivisor(-1,B0), E4);
divcheck("beta_3 = 3e_3 + 3e_4",
  B3 + multdivisor(-1,B0), E3+E3+E3 + E4+E4+E4);
divcheck("gamma_0 = 3e_1 + 3e_2 + e_3 + e_5 + e_6",
  C0 + multdivisor(-1,B0), E1+E1+E1 + E2+E2+E2 + E3 + E5 + E6);
divcheck("gamma_1 = e_5",
  C1 + multdivisor(-1,B0), E5);
divcheck("gamma_2 = 3e_1 + 3e_2 + 3e_3 + 3e_4 + 3e_5 + e_6",
  C2 + multdivisor(-1,B0), E1+E1+E1 + E2+E2+E2
  + E3+E3+E3 + E4+E4+E4 + E5+E5+E5 + E6);
divcheck("gamma_3 = 2e_1 + 2e_2 + e_4 + 3e_5",
  C3 + multdivisor(-1,B0), E1+E1 + E2+E2 + E4 + E5+E5+E5);

// Points of tangency of bitangents
print("");
print("Points of tangency of bitangents");
prdivcheck("2 D_0 = div(x + y + z)", D0 + D0, x + y + z, 1);
prdivcheck("2 D_1 = div(x - y + z)", D1 + D1, x - y + z, 1);
prdivcheck("2 D_2 = div(x + y - z)", D2 + D2, x + y - z, 1);
prdivcheck("2 D_3 = div(x - y - z)", D3 + D3, x - y - z, 1);
prdivcheck("D_0 + D_1 + D_2 + D_3 = div(X^2 + Y^2 + Z^2)",
  D0 + D1 + D2 + D3, x^2 + y^2 + z^2, 1);
prdivcheck("D_1 - D_0 = 2B_1 + 2B_2 - 4B_0 + div(...)",
  D1 + multdivisor(-1,D0) + multdivisor(-2,B1+B2) + multdivisor(4,B0),
  (x - z8*z) * (x - y + z),
  (x^2 + y^2 + z^2) + (d^5 - d^3 - d) * (y^2 - x*z));
prdivcheck("D_2 - D_0 = 2A_1 + 2A_2 + 2B_1 + 2B_2 - 8B_0 + div(...)",
  D2 + multdivisor(-1,D0) + multdivisor(-2,A1+A2+B1+B2) + multdivisor(8,B0),
  (x - z8*z)^2 * (x + y - z),
  (x^2 + y^2 + z^2) * (x + y) - (d^5 - d^3 - d) * z * (x^2 + x*y + y^2));
prdivcheck("D_3 - D_0 = 2A1 + 2A2 - 4B0 + div(...)",
  D3 + multdivisor(-1,D0) + multdivisor(-2,A1+A2) + multdivisor(4,B0),
  (x - z8*z) * (x - y - z),
  (x^2 - y^2 - 2*y*z - z^2) + (d^5 - d^3 - d) * (y^2 + y*z + z^2));
prdivcheck("E = 2B_2 - 2B_0",
  E + multdivisor(-2,B2) + B0+B0,
  1, 1);

divcheck("D1 - D0 = 2e_3 + 2e_4",
  D1 + multdivisor(-1,D0), E3+E3 + E4+E4);
divcheck("D2 - D0 = 2e_1 + 2e_2 + 2e_3 + 2e_4",
  D2 + multdivisor(-1,D0), E1+E1 + E2+E2 + E3+E3 + E4+E4);
divcheck("D3 - D0 = 2e_1 + 2e_2",
  D3 + multdivisor(-1,D0), E1+E1 + E2+E2);
divcheck("E = 2e_4", E, E4+E4);

// Action of sigma_3
print("");
print("Action of sigma_3");
number z8_3 = (d^3)^3; number z4_3 = z8_3^2;
divisor A0_3 = makeDivisor(ideal(x, z8_3^7 * y - z4_3^0 * z), ideal(1));
divisor A1_3 = makeDivisor(ideal(x, z8_3^7 * y - z4_3^1 * z), ideal(1));
divisor A2_3 = makeDivisor(ideal(x, z8_3^7 * y - z4_3^2 * z), ideal(1));
divisor A3_3 = makeDivisor(ideal(x, z8_3^7 * y - z4_3^3 * z), ideal(1));
divisor B0_3 = makeDivisor(ideal(y, z8_3^7 * x - z4_3^0 * z), ideal(1));
divisor B1_3 = makeDivisor(ideal(y, z8_3^7 * x - z4_3^1 * z), ideal(1));
divisor B2_3 = makeDivisor(ideal(y, z8_3^7 * x - z4_3^2 * z), ideal(1));
divisor B3_3 = makeDivisor(ideal(y, z8_3^7 * x - z4_3^3 * z), ideal(1));
divisor C0_3 = makeDivisor(ideal(z, x - z8_3 * z4_3^0 * y), ideal(1));
divisor C1_3 = makeDivisor(ideal(z, x - z8_3 * z4_3^1 * y), ideal(1));
divisor C2_3 = makeDivisor(ideal(z, x - z8_3 * z4_3^2 * y), ideal(1));
divisor C3_3 = makeDivisor(ideal(z, x - z8_3 * z4_3^3 * y), ideal(1));
divisor E1_3 = A1_3 + multdivisor(-1, B0_3);
divisor E2_3 = A2_3 + multdivisor(-1, B0_3);
divisor E3_3 = B1_3 + multdivisor(-1, B0_3);
divisor E4_3 = B2_3 + multdivisor(-1, B0_3);
divisor E5_3 = C1_3 + multdivisor(-1, B0_3);
divisor E6_3 = A1_3 + B1_3 + C1_3 + A2_3 + B2_3 + C2_3 + multdivisor(-6, B0_3);
divisor E_3 =
  multdivisor( 2, makeDivisor(ideal(z8_3^3 * x - z, y), ideal(1)))
+ multdivisor(-2, makeDivisor(ideal(z8_3^7 * x - z, y), ideal(1)));

divcheck("sigma_3(A_0) = A_1", A0_3, A1);
divcheck("sigma_3(A_1) = A_0", A1_3, A0);
divcheck("sigma_3(A_2) = A_3", A2_3, A3);
divcheck("sigma_3(A_3) = A_2", A3_3, A2);
divcheck("sigma_3(B_0) = B_1", B0_3, B1);
divcheck("sigma_3(B_1) = B_0", B1_3, B0);
divcheck("sigma_3(B_2) = B_3", B2_3, B3);
divcheck("sigma_3(B_3) = B_2", B3_3, B2);
divcheck("sigma_3(C_0) = C_1", C0_3, C1);
divcheck("sigma_3(C_1) = C_0", C1_3, C0);
divcheck("sigma_3(C_2) = C_3", C2_3, C3);
divcheck("sigma_3(C_3) = C_2", C3_3, C2);

divcheck("sigma_3(e1) = 2e_1 + e_2 + e_3 + e_4", E1_3, E1+E1 + E2 + E3 + E4);
divcheck("sigma_3(e2) = e_1 + 2e_2 + e_3 + 3e_4", E2_3, E1 + E2+E2 + E3 + E4+E4+E4);
divcheck("sigma_3(e3) = 3e_3", E3_3, E3+E3+E3);
divcheck("sigma_3(e4) = 2e_3 + 3e_4", E4_3, E3+E3 + E4+E4+E4);
divcheck("sigma_3(e5) = 3e_1 + 3e_2 + e_5 + e_6", E5_3, E1+E1+E1 + E2+E2+E2 + E5 + E6);
divcheck("sigma_3(e6) = 2e_3 + e_6", E6_3, E3+E3 + E6);

// Action of sigma_5
print("");
print("Action of sigma_5");
number z8_5 = (d^3)^5; number z4_5 = z8_5^2;
divisor A0_5 = makeDivisor(ideal(x, z8_5^7 * y - z4_5^0 * z), ideal(1));
divisor A1_5 = makeDivisor(ideal(x, z8_5^7 * y - z4_5^1 * z), ideal(1));
divisor A2_5 = makeDivisor(ideal(x, z8_5^7 * y - z4_5^2 * z), ideal(1));
divisor A3_5 = makeDivisor(ideal(x, z8_5^7 * y - z4_5^3 * z), ideal(1));
divisor B0_5 = makeDivisor(ideal(y, z8_5^7 * x - z4_5^0 * z), ideal(1));
divisor B1_5 = makeDivisor(ideal(y, z8_5^7 * x - z4_5^1 * z), ideal(1));
divisor B2_5 = makeDivisor(ideal(y, z8_5^7 * x - z4_5^2 * z), ideal(1));
divisor B3_5 = makeDivisor(ideal(y, z8_5^7 * x - z4_5^3 * z), ideal(1));
divisor C0_5 = makeDivisor(ideal(z, x - z8_5 * z4_5^0 * y), ideal(1));
divisor C1_5 = makeDivisor(ideal(z, x - z8_5 * z4_5^1 * y), ideal(1));
divisor C2_5 = makeDivisor(ideal(z, x - z8_5 * z4_5^2 * y), ideal(1));
divisor C3_5 = makeDivisor(ideal(z, x - z8_5 * z4_5^3 * y), ideal(1));
divisor E1_5 = A1_5 + multdivisor(-1, B0_5);
divisor E2_5 = A2_5 + multdivisor(-1, B0_5);
divisor E3_5 = B1_5 + multdivisor(-1, B0_5);
divisor E4_5 = B2_5 + multdivisor(-1, B0_5);
divisor E5_5 = C1_5 + multdivisor(-1, B0_5);
divisor E6_5 = A1_5 + B1_5 + C1_5 + A2_5 + B2_5 + C2_5 + multdivisor(-6, B0_5);
divisor E_5 =
    multdivisor( 2, makeDivisor(ideal(z8_5^3 * x - z, y), ideal(1)))
  + multdivisor(-2, makeDivisor(ideal(z8_5^7 * x - z, y), ideal(1)));

divcheck("sigma_5(A_0) = A_2", A0_5, A2);
divcheck("sigma_5(A_1) = A_3", A1_5, A3);
divcheck("sigma_5(A_2) = A_0", A2_5, A0);
divcheck("sigma_5(A_3) = A_1", A3_5, A1);
divcheck("sigma_5(B_0) = B_2", B0_5, B2);
divcheck("sigma_5(B_1) = B_3", B1_5, B3);
divcheck("sigma_5(B_2) = B_0", B2_5, B0);
divcheck("sigma_5(B_3) = B_1", B3_5, B1);
divcheck("sigma_5(C_0) = C_2", C0_5, C2);
divcheck("sigma_5(C_1) = C_3", C1_5, C3);
divcheck("sigma_5(C_2) = C_0", C2_5, C0);
divcheck("sigma_5(C_3) = C_1", C3_5, C1);

divcheck("sigma_5(e1) = e_1 + 2e_2 + 2e_3 + 2e_4", E1_5, E1 + E2+E2 + E3+E3 + E4+E4);
divcheck("sigma_5(e2) = 2e_1 + e_2 + 2e_3", E2_5, E1+E1 + E2 + E3+E3);
divcheck("sigma_5(e3) = 3e_3 + 2e_4", E3_5, E3+E3+E3 + E4+E4);
divcheck("sigma_5(e4) = 3e_4", E4_5, E4+E4+E4);
divcheck("sigma_5(e5) = 2e_1 + 2e_2 + 3e_5", E5_5, E1+E1 + E2+E2 + E5+E5+E5);
divcheck("sigma_5(e6) = 2e_4 + e_6", E6_5, E4+E4 + E6);

// Calculation of fixed points
print("");
print("Calculation of fixed points");
divcheck("(sigma_5 - 1)[A0] = 2e_1 + 2e_3 + 3e_4", A2 + multdivisor(-1,A0),
  E1+E1 + E3+E3 + E4+E4+E4);
divcheck("(sigma_3 - 1)[A0] = 3e_1 + 3e_2 + 2e_3 + 3e_4", A1 + multdivisor(-1,A0),
  E1+E1+E1 + E2+E2+E2 + E3+E3 + E4+E4+E4);
divcheck("(sigma_3 sigma_5 - 1)[A0] = 3e_1 + e_2 + 2e_4", A3 + multdivisor(-1,A0),
  E1+E1+E1 + E2 + E4+E4);

// Calculation of Brauer obstruction
print("");
print("Calculation of Brauer obstruction");
prdivcheck("2E = div((X - z8^5 * Z)/(X - z8 * Z)",
  E+E, x - z8^5 * z, x - z8 * z);
prdivcheck("E + sigma_3(E) = div(Y^2/((X - z8 * Z) * (X - z8^3 * Z)))",
  E + E_3, y^2, (x - z8 * z) * (x - z8^3 * z));
prdivcheck("E - sigma_3(E) = div(Y^2/((X - z8 * Z) * (X - z8^7 * Z)))",
  E + multdivisor(-1, E_3), y^2, (x - z8 * z) * (x - z8^7 * z));

quit;
\end{lstlisting}

If the above program is executed successfully,
it outputs the following message.
In the output, ``OK'' means the formula in this paper is correct.

\begin{lstlisting}[basicstyle=\tiny\ttfamily, numbers=left, frame=single]
                     SINGULAR                                 /  Development
 A Computer Algebra System for Polynomial Computations       /   version 4.2.1
                                                           0<
 by: W. Decker, G.-M. Greuel, G. Pfister, H. Schoenemann     \   May 2021
FB Mathematik der Universitaet, D-67653 Kaiserslautern        \
// ** loaded /usr/local/bin/../share/singular/LIB/divisors.lib (4.2.0.1,Jan_2021)
// ** loaded (builtin) customstd.so 
// ** loaded (builtin) gfanlib.so 
// ** loaded /usr/local/bin/../share/singular/LIB/primdec.lib (4.2.0.0,Mar_2021)
// ** loaded /usr/local/bin/../share/singular/LIB/ring.lib (4.1.2.0,Feb_2019)
// ** loaded /usr/local/bin/../share/singular/LIB/absfact.lib (4.1.2.0,Feb_2019)
// ** loaded /usr/local/bin/../share/singular/LIB/triang.lib (4.1.2.0,Feb_2019)
// ** loaded /usr/local/bin/../share/singular/LIB/matrix.lib (4.1.2.0,Feb_2019)
// ** loaded /usr/local/bin/../share/singular/LIB/nctools.lib (4.1.2.0,Feb_2019)
// ** loaded /usr/local/bin/../share/singular/LIB/inout.lib (4.1.2.0,Feb_2019)
// ** loaded /usr/local/bin/../share/singular/LIB/random.lib (4.1.2.0,Feb_2019)
// ** loaded /usr/local/bin/../share/singular/LIB/polylib.lib (4.2.0.0,Dec_2020)
// ** loaded /usr/local/bin/../share/singular/LIB/elim.lib (4.1.2.0,Feb_2019)
// ** loaded /usr/local/bin/../share/singular/LIB/general.lib (4.1.2.0,Feb_2019)

Check linear equivalences of divisors
alpha_0 = 2e_1 + e_2 + 2e_3 + e_4 : OK
alpha_1 = e_1 : OK
alpha_2 = e_2 : OK
alpha_3 = e_1 + 2e_2 + 2e_3 + 4e_4 : OK
beta_1 = e_3 : OK
beta_2 = e_4 : OK
beta_3 = 3e_3 + 3e_4 : OK
gamma_0 = 3e_1 + 3e_2 + e_3 + e_5 + e_6 : OK
gamma_1 = e_5 : OK
gamma_2 = 3e_1 + 3e_2 + 3e_3 + 3e_4 + 3e_5 + e_6 : OK
gamma_3 = 2e_1 + 2e_2 + e_4 + 3e_5 : OK

Points of tangency of bitangents
2 D_0 = div(x + y + z) : OK
2 D_1 = div(x - y + z) : OK
2 D_2 = div(x + y - z) : OK
2 D_3 = div(x - y - z) : OK
D_0 + D_1 + D_2 + D_3 = div(X^2 + Y^2 + Z^2) : OK
D_1 - D_0 = 2B_1 + 2B_2 - 4B_0 + div(...) : OK
D_2 - D_0 = 2A_1 + 2A_2 + 2B_1 + 2B_2 - 8B_0 + div(...) : OK
D_3 - D_0 = 2A1 + 2A2 - 4B0 + div(...) : OK
E = 2B_2 - 2B_0 : OK
D1 - D0 = 2e_3 + 2e_4 : OK
D2 - D0 = 2e_1 + 2e_2 + 2e_3 + 2e_4 : OK
D3 - D0 = 2e_1 + 2e_2 : OK
E = 2e_4 : OK

Action of sigma_3
sigma_3(A_0) = A_1 : OK
sigma_3(A_1) = A_0 : OK
sigma_3(A_2) = A_3 : OK
sigma_3(A_3) = A_2 : OK
sigma_3(B_0) = B_1 : OK
sigma_3(B_1) = B_0 : OK
sigma_3(B_2) = B_3 : OK
sigma_3(B_3) = B_2 : OK
sigma_3(C_0) = C_1 : OK
sigma_3(C_1) = C_0 : OK
sigma_3(C_2) = C_3 : OK
sigma_3(C_3) = C_2 : OK
sigma_3(e1) = 2e_1 + e_2 + e_3 + e_4 : OK
sigma_3(e2) = e_1 + 2e_2 + e_3 + 3e_4 : OK
sigma_3(e3) = 3e_3 : OK
sigma_3(e4) = 2e_3 + 3e_4 : OK
sigma_3(e5) = 3e_1 + 3e_2 + e_5 + e_6 : OK
sigma_3(e6) = 2e_3 + e_6 : OK

Action of sigma_5
sigma_5(A_0) = A_2 : OK
sigma_5(A_1) = A_3 : OK
sigma_5(A_2) = A_0 : OK
sigma_5(A_3) = A_1 : OK
sigma_5(B_0) = B_2 : OK
sigma_5(B_1) = B_3 : OK
sigma_5(B_2) = B_0 : OK
sigma_5(B_3) = B_1 : OK
sigma_5(C_0) = C_2 : OK
sigma_5(C_1) = C_3 : OK
sigma_5(C_2) = C_0 : OK
sigma_5(C_3) = C_1 : OK
sigma_5(e1) = e_1 + 2e_2 + 2e_3 + 2e_4 : OK
sigma_5(e2) = 2e_1 + e_2 + 2e_3 : OK
sigma_5(e3) = 3e_3 + 2e_4 : OK
sigma_5(e4) = 3e_4 : OK
sigma_5(e5) = 2e_1 + 2e_2 + 3e_5 : OK
sigma_5(e6) = 2e_4 + e_6 : OK

Calculation of fixed points
(sigma_5 - 1)[A0] = 2e_1 + 2e_3 + 3e_4 : OK
(sigma_3 - 1)[A0] = 3e_1 + 3e_2 + 2e_3 + 3e_4 : OK
(sigma_3 sigma_5 - 1)[A0] = 3e_1 + e_2 + 2e_4 : OK

Calculation of Brauer obstruction
2E = div((X - z8^5 * Z)/(X - z8 * Z) : OK
E + sigma_3(E) = div(Y^2/((X - z8 * Z) * (X - z8^3 * Z))) : OK
E - sigma_3(E) = div(Y^2/((X - z8 * Z) * (X - z8^7 * Z))) : OK
Auf Wiedersehen.
\end{lstlisting}

\subsection*{Acknowledgements}

The authors would like to thank
Brendan Creutz for comments on
the first version of this paper.
He kindly informed that
he constructed plane quartics over $\Q$
for which
the degree $0$ part of the Picard group is strictly smaller than
the Mordell--Weil group of the Jacobian variety; see \cite[Theorem 6.3, Remark 6.4]{Creutz} for details.

The work of the first author was supported by JSPS KAKENHI Grant Number
16K17572 and 21K13773.
The work of the second author was supported by JSPS
KAKENHI Grant Number 20674001, 26800013, 21H00973.
The work of the third author was
supported by JSPS KAKENHI Grant Number 
18H05233 and 20K14295. 

In the course of writing this paper,
we calculated linear equivalences of divisors on
the quartic $X^4 + Y^4 + Z^4 = 0$
using Singular \cite{Singular} and \texttt{divisors.lib} \cite{Singular:DivisorsLib}.

\bibliographystyle{amsplain}

\begin{thebibliography}{99}
\bibitem{Beauville}
  Beauville, A.,
  \textit{Determinantal hypersurfaces},
  Michigan Math.\ J.\ \textbf{48} (2000), 39-64.

\bibitem{BhargavaGross}
  Bhargava, M., Gross, B.\ H.,
  \textit{Arithmetic invariant theory},
  Symmetry:\ representation theory and its applications,
  33-54, Progr.\ Math., \textbf{257}, Birkh\"auser/Springer, New York, 2014.

\bibitem{BoschLuetkebohmertRaynaud}
  Bosch, S., L\"utkebohmert, W., Raynaud, M., \textit{N\'eron models},
  Ergebnisse der Mathematik und ihrer Grenzgebiete (3), \textbf{21}.
  Springer-Verlag, Berlin, 1990.

\bibitem{CiperianiKrashen}
  Ciperiani, M., Krashen, D.,
  \textit{Relative Brauer groups of genus 1 curves},
  Israel J.\ Math.\ \textbf{192} (2012), no.\ 2, 921-949.

\bibitem{Creutz}
  Creutz, B.,
  \textit{Generalized Jacobians and explicit descents},
  Math.\ Comp.\ \textbf{89} (2020), no.\ 323, 1365-1394.

\bibitem{Dolgachev}
  Dolgachev, I. V.,
  \textit{Classical algebraic geometry, A modern view},
  Cambridge University Press, Cambridge, 2012.

\bibitem{Edge}
  Edge, W.\ L.,
  \textit{Determinantal representations of $x^4 + y^4 + z^4$},
  Math.\ Proc.\ Cambridge Phil.\ Soc.\ \textbf{34} (1938), 6-21.

\bibitem{Faddeev}
  Faddeev, D.\ K., \textit{Group of divisor classes on the curve
  defined by the equation $x^4 + y^4 = 1$},
  Soviet Math.\ Dokl.\ \textbf{1} (1960), 1149-1151
  (Dokl.\ Akad.\ Nauk SSSR \textbf{134} (1960), 776-777 (Russian original)).

\bibitem{Ishitsuka:Proportion}
  Ishitsuka, Y.,
  \textit{A positive proportion of cubic curves over $\Q$ admit linear determinantal representations},
  J.\ Ramanujan Math.\ Soc.\ \textbf{32} (2017), no.\ 3, 231-257.

\bibitem{IshitsukaIto:SDR}
  Ishitsuka, Y., Ito, T.,
  \textit{The local-global principle for symmetric determinantal representations of smooth plane curves},
  Ramanujan J.\ \textbf{43} (2017), no.\ 1, 141-162.

\bibitem{IshitsukaIto:SDRChar2}
  Ishitsuka, Y., Ito, T.,
  \textit{The local-global principle for symmetric determinantal representations of smooth plane curves in characteristic two},
  J.\ Pure Appl.\ Algebra \textbf{221} (2017), no.\ 6, 1316-1321.

\bibitem{IshitsukaItoOhshita:LDR}
  Ishitsuka, Y., Ito, T., Ohshita, T.,
  \textit{On algorithms to obtain linear determinantal representations of smooth plane curves of higher degree},
  JSIAM Lett.\ \textbf{11} (2019), 9-12.

\bibitem{IshitsukaItoOhshita:FermatQuartic}
  Ishitsuka, Y., Ito, T., Ohshita, T.,
  \textit{Explicit calculation of the mod $4$ Galois representation associated with the Fermat quartic},
  Int.\ J.\ Number Theory \textbf{16} (2020), no.\ 4, 881-905.

\bibitem{IIOTU}
  Ishitsuka, Y., Ito, T., Ohshita, T.,
  Taniguchi, T., Uchida, Y.,
  \textit{The local-global property for bitangents of plane quartics},
  JSIAM Lett.\ \textbf{12} (2020), 41-44.

\bibitem{Jarden}
  Jarden, M.,
  \textit{Algebraic patching},
  Springer Monographs in Mathematics. Springer, Heidelberg, 2011.

\bibitem{Pop}
  Pop, F.,
  \textit{Embedding problems over large fields},
  Ann.\ of Math.\ (2) \textbf{144} (1996), no.\ 1, 1-34.

\bibitem{Rohrlich}
  Rohrlich, D.\ E.,
  \textit{Points at infinity on the Fermat curves},
  Invent.\ Math.\ \textbf{39} (1977), no.\ 2, 95-127.

\bibitem{Singular}
  Decker, W., Greuel, G.-M., Pfister, G., Sch{\"o}nemann, H.,
  \textit{Singular} (Version 4-2-1) --- A computer algebra system for polynomial computations,
  2021, \url{http://www.singular.uni-kl.de}.

\bibitem{Singular:DivisorsLib}
  Boehm, J., Kastner, L., Lorenz, B., Sch{\"o}nemann, H. Ren, Y.,
  {\tt divisors.lib},
  A Singular library for computing divisors and P-Divisors
  (version 4.2.0.0), 2021.
\end{thebibliography}

\end{document}